\documentclass[11pt,a4paper]{article}
\usepackage[titletoc,title]{appendix}
\usepackage{amsmath,amsfonts,amssymb,bm,hyperref,amsthm,tikz,amscd}

\numberwithin{equation}{section}
\newtheorem{theorem}{Theorem}[section]
\newtheorem{proposition}[theorem]{Proposition}
\newtheorem{lemma}[theorem]{Lemma}
\newtheorem{corollary}[theorem]{Corollary}
\theoremstyle{definition}

\newtheorem{remark}[theorem]{Remark}
\newtheorem{definition}[theorem]{Definition}

\usepackage[top=1in, bottom=1in, left=1in, right=1in]{geometry}
\newcommand{\Ad}{\operatorname{Ad}}

\newcommand{\spin}{\operatorname{spin}}

\newcommand{\RP}{{\mathbb R}{\rm P}}

\newcommand{\HP}{{\mathbb H}{\rm P}}
\renewcommand{\tilde}{\widetilde}

\newcommand{\largewedge}{\mbox{\Large $\wedge$}}

\usepackage{enumerate}

\usepackage{tikz}
\usetikzlibrary{shapes,arrows}

\usepackage[all]{xy}

\usepackage[normalem]{ulem}

\begin{document}

\title{Classification of first order sesquilinear forms}
\author{
Matteo Capoferri
\thanks{MC:
Department of Mathematics,
University College London,
Gower Street,
London WC1E~6BT,
\ UK;
matteo.capoferri.16@ucl.ac.uk,
\url{https://www.ucl.ac.uk/\~ucahmca/};
}
\and
Nikolai Saveliev\thanks{NS:
Department of Mathematics,
University of Miami,
PO Box 249085
Coral Gables,
FL 33124,
USA;
saveliev@math.miami.edu,
\url{http://www.math.miami.edu/\~saveliev/};
NS was supported by
LMS grant 21420,  EPSRC grant EP/M000079/1, and the Simons Collaboration Grant 426269.
}
\and
Dmitri Vassiliev\thanks{DV:
Department of Mathematics,
University College London,
Gower Street,
London WC1E~6BT,
UK;
d.vassiliev@ucl.ac.uk,
\url{https://www.ucl.ac.uk/\~ucahdva/};
DV was supported by EPSRC grant EP/M000079/1.
}}

\renewcommand\footnotemark{}


\maketitle
\begin{abstract}
A natural way to obtain a system of
partial differential equations on a manifold is to vary a
suitably defined sesquilinear form.
The sesquilinear forms we study are Hermitian forms acting on sections of the trivial $\mathbb{C}^n$-bundle over a smooth $m$-dimensional manifold without boundary. More specifically, we are concerned with first order sesquilinear forms, namely, those generating first order systems.
Our goal is to classify such forms up to $GL(n,\mathbb{C})$ gauge equivalence.
We achieve this classification in the special case of $m=4$ and $n=2$
by means of geometric and topological invariants (e.g.\ Lorentzian metric, spin/spin$^c$ structure, electromagnetic covector potential)
naturally contained within the sesquilinear form -- a purely analytic object.
Essential to our approach is the
interplay of techniques from analysis, geometry, and topology.

\

{\bf Keywords:}
sesquilinear forms, first order systems, gauge transformations, $\operatorname{spin}^c$ structure.

\

{\bf MSC classes:}
primary 35F35; secondary 35L40, 35R01, 53C27.










\end{abstract}

\newpage

\tableofcontents

\newpage

\section{Introduction}
\label{Introduction}

In this paper we study sesquilinear forms of a particular type,
namely, those that generate first order systems
of partial differential equations on manifolds.

In order to provide motivation for our analysis,
let us first discuss some basic facts from linear algebra
in finite dimension.

Working in a $k$-dimensional complex vector space $V$, consider an Hermitian form
\[
S:V\times V\to\mathbb{C}, \quad (u,v)\mapsto S(u,v).
\]
Here $S$ is assumed to be antilinear in the first argument and linear in the second.
Variation of the real-valued action $S(v,v)$
produces the following linear field equation for $v$:
\begin{equation}
\label{null space of the of sesquilinear form}
S(u,v)=0,\qquad\forall u\in V.
\end{equation}

Suppose now that our vector space $V$ is equipped with an additional structure,
an inner product $\langle\,\cdot\,,\,\cdot\,\rangle$.
Then the sesquilinear form $S$ and  inner product $\langle\,\cdot\,,\,\cdot\,\rangle$
uniquely define a self-adjoint linear operator $L:V\to V$ via the formula
\begin{equation}
\label{relation between form and operator}
S(u,v)=\langle u,Lv\rangle,\qquad\forall u,v\in V.
\end{equation}
The argument also works the other way round: a self-adjoint linear operator
uniquely defines  an Hermitian sesquilinear form via formula \eqref{relation between form and operator}.
Thus, in an inner product space the concepts of
Hermitian sesquilinear form
and
self-adjoint linear operator
are equivalent.

Given a linear operator $L$, we can consider the linear equation
\begin{equation}
\label{null space of the of operator}
Lv=0.
\end{equation}
If $S$ and $L$ are related as in \eqref{relation between form and operator},
then equations
\eqref{null space of the of sesquilinear form}
and 
\eqref{null space of the of operator}
are equivalent.

It may seem that there is no point in working with Hermitian sesquilinear forms
and that one can work with self-adjoint linear operators instead, which would be easier for practical purposes.
However, there is a point because the statement regarding the equivalence
of linear equations
\eqref{null space of the of sesquilinear form}
and 
\eqref{null space of the of operator}
is based on the use of an inner product.
The concept of an Hermitian sesquilinear form
is more fundamental than the concept of a self-adjoint linear operator
in that it does not require an inner product for its definition.
One can formulate and study the linear equation 
\eqref{null space of the of sesquilinear form}
without introducing an inner product.

In the class of problems we are interested in, the above toy
model translates into the study of partial differential equations
on manifolds in a setting when there is no natural definition of
an inner product invariant under relevant gauge transformations.
Such a situation arises, for instance, when dealing with physically meaningful problems in 4-dimensional Lorentzian spacetime, see Sections~\ref{Sesquilinear forms vs linear operators} and~\ref{Applications}. Fully relativistic equations of mathematical physics are not always associated with a natural inner product,
not even an indefinite non-degenerate one.

This is why studying sesquilinear forms and their classification is an interesting mathematical problem with relevant applications. More precisely, the goal of our paper is to study and classify sesquilinear forms acting on compactly supported smooth sections of the trivial $\mathbb{C}^n$-bundle over a smooth manifold $M$, whose coordinate representation involves the sections themselves and their first derivatives but no products of first derivatives. Adopting a non-canonical approach, we ask the question: when do two sesquilinear forms written in their coordinate representation correspond to the same abstract sesquilinear form? In other words, we are interested in establishing when two sesquilinear forms can be obtained one from the other by a pointwise change of basis in the fibre depending smoothly on the base point. As it turns out, this problem can be solved thanks to the interplay of techniques from algebraic topology, geometry and analysis of partial differential equations.


\

Our paper is structured as follows.

In Section~\ref{First order sesquilinear forms} we provide a precise definition of the class of sesquilinear forms we work with using the language of analysis of partial differential equations. 

In Section~\ref{Statement of the problem} we formulate the mathematical problem we want to address, namely, the classification of first order sesquilinear forms, distinguishing the two different types of classification we will be looking at. 

Section~\ref{Main results} contains a brief description of the main result of the paper: our classification theorems in dimension four.

Sections~\ref{Invariant objects encoded within sesquilinear forms} and~\ref{Transition from analysis to topology} comprise preparatory work towards the proof of the main theorems. In Section~\ref{Invariant objects encoded within sesquilinear forms} we analyse properties of sesquilinear forms, identifying geometric and topological objects naturally encoded in their analytic definition. In Section~\ref{Transition from analysis to topology} we recast our analytic definition of equivalence of sesquilinear forms in a purely algebraic topological fashion, proving the equivalence of the two formulations. 

Our main theorems are proved in Section~\ref{Proofs of main theorems}.

Section~\ref{The 3-dimensional Riemannian case} is concerned with a similar analysis in dimension three, under suitable additional conditions. We also examine two explicit examples.

In Section~\ref{Sesquilinear forms vs linear operators} we revisit the sesquilinear forms vs linear operators issue in the context of our main results. 

In conclusion, in Section~\ref{Applications} we briefly mention some physically meaningful applications of our results.

The paper is complemented by
Appendix~\ref{The concepts of principal and subprincipal symbol}
where we explain the relation between the traditional definitions
of symbols of (pseudo)differential
operators and our definitions for sesquilinear forms.

\section{First order sesquilinear forms}
\label{First order sesquilinear forms}

Let $M$ be a real connected smooth $m$-manifold without boundary, not necessarily compact.
Local coordinates on $M$ will be denoted by $x^\alpha$, $\alpha=1,\dots,m$.

We will be working with compactly supported smooth
functions $u:M\to\mathbb{C}^n$.
Such functions can be thought of as
sections of the trivial
$\mathbb{C}^n$-bundle over $M$
or as
$n$-columns of smooth complex-valued scalar fields.
They form an (infinite-dimensional) vector space
$C_0^\infty(M,\mathbb{C}^n)$.

\begin{definition}
\label{definition of sequilinear form}
A first order sesquilinear form is a functional
\begin{equation}
\label{definition of sequilinear 1}
S(u,v):=\int_M
\left[
u^*\,\mathbf{A}^\alpha\,v_{x^\alpha}
+
u_{x^\alpha}^*\,\mathbf{B}^\alpha\,v
+
u^*\,\mathbf{C}\,v
\right]dx\,,
\qquad u,v\in C_0^\infty(M,\mathbb{C}^n),
\end{equation}
where
$\mathbf{A}^\alpha(x)$, $\mathbf{B}^\alpha(x)$ and $\mathbf{C}(x)$
are some prescribed smooth complex $n\times n$ matrix-functions,
the subscript $x^\alpha$ indicates partial differentiation,
the star stands for Hermitian conjugation (transposition and complex conjugation)
and $dx=dx^1\dots dx^m$.
We adopt the summation convention
over repeated indices.
\end{definition}

In formula \eqref{definition of sequilinear 1}
the elements of the matrix-function $\mathbf{C}$ are densities,
whereas the elements of the matrix-functions $\mathbf{A}$ are $\mathbf{B}$ are vector densities.
Here and further on we use bold script for density-valued quantities.

Performing integration by parts, one can rewrite the sesquilinear form
\eqref{definition of sequilinear 1} in many different ways.
We define the canonical representation of a first order sesquilinear form as
\begin{equation}
\label{definition of sequilinear 2}
S(u,v)=\int_M
\left[
-\frac i2\,
u^*\,\mathbf{E}^\alpha\,v_{x^\alpha}
+\frac i2
u_{x^\alpha}^*\,\mathbf{E}^\alpha\,v
+
u^*\,\mathbf{F}\,v
\right]dx\,.
\end{equation}
The matrix-functions in
\eqref{definition of sequilinear 1}
and
\eqref{definition of sequilinear 2}
are related by formulae
\[
\mathbf{E}^\alpha
=
i(\mathbf{A}^\alpha-\mathbf{B}^\alpha),
\qquad
\mathbf{F}=\mathbf{C}-\frac12\frac{\partial(\mathbf{A}^\alpha+\mathbf{B}^\alpha)}{\partial x^\alpha}\,.
\]
Recall the well-known fact
that if $\mathbf{w}^\alpha$ is a vector density then
$\partial\mathbf{w}^\alpha/\partial x^\alpha$ is a density,
so elements of the matrix-function $\mathbf{F}(x)$ are densities.

We define the principal, subprincipal and full symbols of the sesquilinear form
\eqref{definition of sequilinear 2}
as
\begin{equation}
\label{definition of density-valued principal symbol}
\mathbf{S}_\mathrm{prin}(x,p)
:=
\mathbf{E}^\alpha(x)\,p_\alpha\,,
\end{equation}
\begin{equation}
\label{definition of density-valued subprincipal symbol}
\mathbf{S}_\mathrm{sub}(x)
:=
\mathbf{F}(x),
\end{equation}
\begin{equation}
\label{definition of density-valued full symbol}
\mathbf{S}_\mathrm{full}(x,p)
:=
\mathbf{S}_\mathrm{prin}(x,p)
+
\mathbf{S}_\mathrm{sub}(x)
\end{equation}
respectively.
Here $p_\alpha$, $\alpha=1,\dots,m$, is the dual variable (momentum)
and all the above symbols are well defined on the cotangent bundle $T^*M$.
It is easy to see that the full symbol
uniquely determines our first order sesquilinear form
and that our sesquilinear form is Hermitian
(that is, $S(u,v)=\overline{S(v,u)}\,$)
if and only if its full symbol is Hermitian.

Establishing a correspondence between a sesquilinear form or a (pseudo)differential operator on the one hand
and a (full) symbol on the other hand is often referred to as \emph{quantisation}.
The argument in the above paragraph shows that first order sesquilinear forms admit
a particularly convenient and natural quantisation.

Further on we work with Hermitian first order sesquilinear forms.

An Hermitian first order sesquilinear form $S(u,v)$ defines a real-valued action
$S(v,v)$.
Variation of this action produces field equations for $v$.
This is a system of $n$ linear scalar first order partial differential equations for $n$ unknown
complex-valued scalar fields.
Our interest in such systems is the motivation for the current paper.

Note that, according to our Definition \ref{definition of sequilinear form},
a first order sesquilinear form does not contain the term
$\,u_{x^\alpha}^*\,
\mathbf{D}^{\alpha\beta}
\,v_{x^\beta}\,$.
The presence of such a term would fundamentally change the corresponding
field equations, making them second order,
whereas we are interested in first order systems.

\begin{definition}
\label{definition of non-degeneracy}
We say that the sesquilinear form $S$ is \emph{non-degenerate} if
\begin{equation}
\label{definition of non-degeneracy equation}
\mathbf{S}_\mathrm{prin}(x,p)\ne0,\qquad\forall(x,p)\in T^*M\setminus\{0\}.
\end{equation}
\end{definition}

Condition \eqref{definition of non-degeneracy equation}
means that $\mathbf{S}_\mathrm{prin}$ does not vanish as a matrix,
i.e.~for any $(x,p)\in T^*M\setminus\{0\}$ the matrix
$\mathbf{S}_\mathrm{prin}(x,p)$ has at least one nonzero element.
This is the weakest possible non-degeneracy condition.

Further on we work with non-degenerate Hermitian first order sesquilinear forms.

\section{Statement of the problem}
\label{Statement of the problem}

\subsection{General linear classification}
\label{General linear classification}
Consider a smooth matrix-function
\begin{equation}
\label{matrix-function GL}
R:M\to GL(n,\mathbb{C}).
\end{equation}
Given a sesquilinear form
\eqref{definition of sequilinear 2}
we can now define another sesquilinear form
\begin{equation}
\label{sesquilinear form with tilde equation 1}
\tilde S(u,v):=S(Ru,Rv).
\end{equation}
We interpret this new sesquilinear form as a different representation of
our original sesquilinear form. What we did is we changed, fibrewise, the basis
in our $\mathbb{C}^n$-bundle over $M$ using the gauge transformation $R$.

The explicit formula for $\tilde S(u,v)$ reads
\begin{equation*}
\label{sesquilinear form with tilde equation 2}
\tilde S(u,v)=\int_M
\left[
-\frac i2\,
u^*\,\tilde{\mathbf{E}}^\alpha\,v_{x^\alpha}
+\frac i2
u_{x^\alpha}^*\,\tilde{\mathbf{E}}^\alpha\,v
+
u^*\,\tilde{\mathbf{F}}\,v
\right]dx\,,
\end{equation*}
where
\begin{equation*}
\label{sesquilinear form with tilde equation 3}
\tilde{\mathbf{E}}^\alpha=R^*\,\mathbf{E}^\alpha R,
\qquad
\tilde{\mathbf{F}}=R^*\,\mathbf{F}R
+\frac i2
\left[
R^*_{x^\alpha}\,\mathbf{E}^\alpha R
-
R^*\,\mathbf{E}^\alpha R_{x^\alpha}
\right].
\end{equation*}
The corresponding full symbol is
\begin{equation}
\label{sesquilinear form with tilde equation 4}
\tilde{\mathbf{S}}_\mathrm{full}
=
R^*\,\mathbf{S}_\mathrm{full}\,R
+\frac i2
\left[
R^*_{x^\alpha}(\mathbf{S}_\mathrm{full})_{p_\alpha}R
-
R^*(\mathbf{S}_\mathrm{full})_{p_\alpha}R_{x^\alpha}
\right].
\end{equation}

Our goal is to perform the above argument the other way round, solving, effectively,
an `inverse problem'.
Namely, suppose we are given two full symbols,
$\mathbf{S}_\mathrm{full}(x,p)$ and $\tilde{\mathbf{S}}_\mathrm{full}(x,p)$.
Do they describe the same sesquilinear form?
In order to deal with this question rigorously we introduce the following definition.

\begin{definition}
\label{definition of GL equivalent sesquilinear forms}
We say that
two full symbols
$\mathbf{S}_\mathrm{full}(x,p)$ and $\tilde{\mathbf{S}}_\mathrm{full}(x,p)$
are $GL$-\emph{equivalent} if
there exists a smooth matrix-function \eqref{matrix-function GL}
such that
\eqref{sesquilinear form with tilde equation 4}
is satisfied.
\end{definition}

\subsection{Special linear classification}
\label{Special linear classification}

We will also deal with the problem of equivalence of symbols
in a more restrictive, special linear setting.

\begin{definition}
\label{definition of SL equivalent sesquilinear forms}
We say that
two full symbols
$\mathbf{S}_\mathrm{full}(x,p)$ and $\tilde{\mathbf{S}}_\mathrm{full}(x,p)$
are $SL$-\emph{equivalent} if
there exists a smooth matrix-function
\begin{equation}
\label{matrix-function SL}
R:M\to SL(n,\mathbb{C})
\end{equation}
such that
\eqref{sesquilinear form with tilde equation 4}
is satisfied.
\end{definition}

We now explain the motivation for Definition
\ref{definition of SL equivalent sesquilinear forms}.



Suppose that we have an additional structure in our mathematical model,
a complex-valued volume form,
namely,
a non-vanishing map 
\[
\operatorname{vol}:M \to {\largewedge}^{n,0} (\mathbb{C}^n),
\qquad
\operatorname{vol}(x)=c(x)\,dz^1\wedge\ldots\wedge dz^n,
\]
where $c(x)$ is some prescribed smooth non-vanishing complex scalar field.

The transformation $u\to Ru$,
where $R$ is a matrix-function \eqref{matrix-function GL},
turns $\operatorname{vol}$ into the complex-valued volume form
$
\widetilde{\operatorname{vol}}(x)=\tilde c(x)\,dz^1\wedge\ldots\wedge dz^n
$
with $\tilde c(x)=c(x)\,\det R(x)$.

As in the previous subsection, we consider the `inverse' problem
which now involves both the sesquilinear form
and the complex-valued volume form.
Namely, consider two symbols
$\mathbf{S}_\mathrm{full}(x,p)$ and $\tilde{\mathbf{S}}_\mathrm{full}(x,p)$
and two non-vanishing scalar fields
$c(x)$ and $\tilde c(x)$.
Does there exist a smooth matrix-function
\eqref{matrix-function GL} which turns
$(\mathbf{S}_\mathrm{full},c)$
into
$(\tilde{\mathbf{S}}_\mathrm{full},\tilde c)$\,?

One way of addressing the above question is as follows.
Choose an arbitrary smooth matrix-function $Q:M\to GL(n,\mathbb{C})$
such that $\det Q(x)=c(x)/\tilde c(x)$
(for example, one can take
\linebreak
$Q(x)=\operatorname{diag}\,(c(x)/\tilde c(x),1,\ldots,1)$\,)
and view the sesquilinear form $\tilde{\mathbf{S}}(Qu,Qv)$ as the `new'
sesquilinear form $\tilde{\mathbf{S}}$.
The two complex-valued volume forms now have the same representation.
After this we can only apply
$SL(n,\mathbb{C})$-transformations \eqref{matrix-function SL}
to establish whether the two sesquilinear forms
$\mathbf{S}_\mathrm{full}(x,p)$ and $\tilde{\mathbf{S}}_\mathrm{full}(x,p)$
are equivalent,
because we do not want to change the complex-valued volume form.
This reduces the problem to checking whether the symbols
are $SL$-equivalent in the sense of
Definition~\ref{definition of SL equivalent sesquilinear forms}.

Alternatively, we can do the argument the other way round.
Take an arbitrary smooth matrix-function $Q:M\to GL(n,\mathbb{C})$
such that $\det Q(x)=\tilde c(x)/c(x)$
and view the sesquilinear form $\mathbf{S}(Qu,Qv)$ as the `new'
sesquilinear form $\mathbf{S}$ etc.


It is easy to see that the outcome of this exercise
does not depend on which way we proceed or which $Q$ we choose.
In group-theoretic language,
this corresponds to the fact that
the group of matrix-functions
\eqref{matrix-function SL}
is a normal subgroup 
of the group of matrix-functions \eqref{matrix-function GL}.
The matrix-function $Q$ picks a particular element in each of
the left cosets (or, equivalently, right cosets)
of
$C^\infty(M,GL(n,\mathbb{C}))/C^\infty(M,SL(n,\mathbb{C}))$.

\subsection{Gauge transformations}
The transformations of symbols (and corresponding sesquilinear forms)
described in this section can be interpreted as gauge transformations.
Our gauge transformations come in two versions,
general linear (subsection \ref{General linear classification})
and special linear (subsection \ref{Special linear classification}).
In the current paper we do not dwell on the physical meaning of our  gauge transformations
and pursue our analysis from a purely mathematical standpoint.

\section{Main results}
\label{Main results}

The main problem addressed in the current paper is to give necessary and sufficient
conditions for a pair of full symbols to be
$GL$-equivalent or $SL$-equivalent.
Our explicit non-canonical approach will eventually produce a full classification
of equivalence classes of sesquilinear forms for the special case
\begin{equation}
\label{m equals 4 n equals 2}
m=4,\qquad n=2,
\end{equation}
i.e.~the case when we are dealing with a pair of complex-valued scalar fields
over a 4-manifold.

Under the assumption \eqref{m equals 4 n equals 2} we have the
following two theorems, which are our main results.

\begin{theorem}
\label{main theorem GL}
Let $M$ be a connected 4-manifold without boundary and let $S$ and $\widetilde{S}$ be non-degenerate sesquilinear forms acting on sections of the trivial $\mathbb{C}^2$-bundle over $M$, see Definition~\ref{definition of non-degeneracy}. Then the corresponding full symbols
$\mathbf{S}_\mathrm{full}(x,p)$ and $\,\tilde{\mathbf{S}}_\mathrm{full}(x,p)$
are $GL$-\emph{equivalent} if and only if
\begin{enumerate}[(i)]
\itemsep0em
\item
the metrics encoded within these symbols belong to the same conformal class,
\item
the electromagnetic covector potentials encoded within these symbols
belong to the same cohomology class in $H^1_{\mathrm{dR}}(M)$,
\item
their topological charges are the same,
\item
their temporal charges are the same and
\item
they have the same 2-torsion $\spin^c$ structure.
\end{enumerate}
\end{theorem}

\begin{theorem}
\label{main theorem SL}
Let $M$ be a connected 4-manifold without boundary and let $S$ and $\widetilde{S}$ be non-degenerate sesquilinear forms acting on sections of the trivial $\mathbb{C}^2$-bundle over $M$, see Definition~\ref{definition of non-degeneracy}. Then the corresponding full symbols
$\mathbf{S}_\mathrm{full}(x,p)$ and $\,\tilde{\mathbf{S}}_\mathrm{full}(x,p)$
are $SL$-\emph{equivalent} if and only if
\begin{enumerate}[(i)]
\itemsep0em
\item
the metrics encoded within these symbols are the same,
\item
the electromagnetic covector potentials encoded within these symbols are the same,
\item
their topological charges are the same,
\item
their temporal charges are the same and
\item
they have the same spin structure.
\end{enumerate}
\end{theorem}

The geometric and topological objects appearing in
(i)--(v) in Theorems \ref{main theorem GL} and \ref{main theorem SL}
will be introduced in
Section~\ref{Invariant objects encoded within sesquilinear forms}
and examined further in Section~\ref{Transition from analysis to topology}.
The proof of the above theorems will be given in Section~\ref{Proofs of main theorems}.

The construction presented in
Sections~\ref{Invariant objects encoded within sesquilinear forms}--\ref{Proofs of main theorems}
is not straightforward and comes in several steps
which combine techniques from differential geometry, algebraic topology
and analysis of partial differential equations.

\begin{remark}
The assumption \eqref{m equals 4 n equals 2} may look quite restrictive at first glance.
However, there are reasons for restricting the dimension of the manifold and number of
scalar fields --- reasons to do with the existence or non-existence of an inner product
compatible with gauge transformations.
See Section~\ref{Sesquilinear forms vs linear operators} and last paragraph of Appendix~\ref{The concepts of principal and subprincipal symbol} for further details.
\end{remark}

\section{Invariant objects encoded within sesquilinear forms}
\label{Invariant objects encoded within sesquilinear forms}

\subsection{Geometric objects}
\label{Geometric objects}

Let us first explain why the case
\eqref{m equals 4 n equals 2} is special.

We start by observing that having the weaker constraint
\begin{equation}
\label{m equals n squared}
m=n^2
\end{equation}
already brings about important geometric consequences.
Namely, under the condition \eqref{m equals n squared}
a manifold $M$ admits a non-degenerate Hermitian first order sesquilinear form
if and only if it is parallelizable.
The proof of this statement retraces that of
\cite[Lemma 1.2]{jmp2017}.

Further on we assume that our manifold $M$ is parallelizable.
Without this assumption we would not have any
non-degenerate Hermitian first order sesquilinear forms
to work with.

Setting $n=2$ and $m=4$ has even more profound geometric consequences.
Namely, observe that the determinant of the principal symbol is a
quadratic form in momentum $p$:
\begin{equation}
\label{definition of metric density}
\det\mathbf{S}_\mathrm{prin}(x,p)=
-\mathbf{g}^{\alpha\beta}(x)\,p_\alpha p_\beta\,,
\end{equation}
where $\mathbf{g}^{\alpha\beta}(x)$ is a real symmetric
$4\times4$
matrix-function with values in 2-densities.
More precisely, $\mathbf{g}$ is a rank two symmetric tensor density
of weight two.

The quadratic form
$\mathbf{g}^{\alpha\beta}$
has Lorentzian signature, i.e.~it has three positive eigenvalues and one negative
eigenvalue,
see \cite[Lemma 2.1]{nongeometric}.
This implies, in particular, that
\begin{equation*}
\label{metric density is non-degenerate}
\det\mathbf{g}^{\alpha\beta}(x)<0,
\qquad\forall x\in M.
\end{equation*}

Put
\begin{equation}
\label{initial definition of Lorentzian density}
\rho(x):=(-\det\mathbf{g}^{\mu\nu}(x))^{1/6}.
\end{equation}
The quantity \eqref{initial definition of Lorentzian density}
is a density. This observation allows us to define the Lorentzian metric
\begin{equation}
\label{definition of metric}
g^{\alpha\beta}(x):=
(\rho(x))^{-2}\ \mathbf{g}^{\alpha\beta}(x).
\end{equation}
Of course, formula
\eqref{initial definition of Lorentzian density}
can now be rewritten in more familiar form as
\begin{equation*}
\label{more familiar definition of Lorentzian density}
\rho(x)=(-\det g_{\mu\nu}(x))^{1/2}.
\end{equation*}

We see that the case
\eqref{m equals 4 n equals 2}
is special in that there is a Lorentzian metric encoded
within our sesquilinear form.
This Lorentzian metric $g$ is defined by the
explicit formulae
\eqref{definition of metric},
\eqref{initial definition of Lorentzian density},
\eqref{definition of metric density}.

Let $g^{\alpha\beta}$ be the contravariant metric
tensor encoded within the sesquilinear for $S$.
Then the contravariant metric
tensor $\tilde g^{\alpha\beta}$ encoded within the sesquilinear form $\tilde S$
defined by \eqref{sesquilinear form with tilde equation 1} is
\begin{equation}
\label{conformal scaling of metric under gauge transformations}
\tilde g^{\alpha\beta}=|\det R|^{-2/3}\,g^{\alpha\beta}.
\end{equation}
We see that the metric transforms conformally under
the action of $R$ as in \eqref{matrix-function GL}.
In particular, it is invariant under \eqref{matrix-function SL}.


The second geometric object encoded within our
sesquilinear form is the electromagnetic covector potential.
In order to single it out
we first introduce the concept of covariant sub\-principal symbol
\begin{equation}
\label{covariant subprincipal symbol}
\mathbf{S}_\mathrm{csub}:=
\mathbf{S}_\mathrm{sub}
+\frac i{16}\,
\mathbf{g}_{\alpha\beta}
\{
\mathbf{S}_\mathrm{prin}
,
\operatorname{adj}\mathbf{S}_\mathrm{prin}
,
\mathbf{S}_\mathrm{prin}
\}_{p_\alpha p_\beta},
\end{equation}
where $\mathbf{g}_{\alpha\beta}$ is the inverse of $\mathbf{g}^{\alpha\beta}$,
\begin{equation*}
\label{definition of generalised Poisson bracket}
\{F,G,H\}:=F_{x^\alpha}GH_{p_\alpha}-F_{p_\alpha}GH_{x^\alpha}
\end{equation*}
is the generalised Poisson bracket on matrix-functions
and $\,\operatorname{adj}\,$ is the operator of matrix adjugation
\begin{equation*}
\label{definition of adjugation}
F=\begin{pmatrix}a&b\\ c&d\end{pmatrix}
\mapsto
\begin{pmatrix}d&-b\\-c&a\end{pmatrix}
=:\operatorname{adj}F.
\end{equation*}
We define the electromagnetic covector potential $A$ as
the --- unique, due to \eqref{definition of non-degeneracy equation} --- real-valued solution of
\begin{equation}
\label{definition of electromagnetic covector potential}
\mathbf{S}_\mathrm{csub}(x)=\mathbf{S}_\mathrm{prin}(x,A(x)).
\end{equation}
Note that \eqref{definition of electromagnetic covector potential}
is a system of four linear algebraic equations for the
four components of~$A$.

\begin{lemma}
The electromagnetic covector potential is given explicitly by the
following formula
\begin{equation}
\label{explicit formula for A}
A_\alpha
=-\frac12\,\mathbf{g}_{\alpha\beta}
\operatorname{tr}
\left(
(\mathbf{S}_\mathrm{csub})
\,
(\operatorname{adj}(\mathbf{S}_\mathrm{prin})_{p_\beta})
\right).
\end{equation}
\end{lemma}

\begin{proof}
In view of \eqref{definition of metric density}, multiplication of both sides of \eqref{definition of electromagnetic covector potential} by $\operatorname{adj}\,(\mathbf{S}_\mathrm{prin}(x,p))$ gives
\begin{equation}\label{lemma A temp 1}
\begin{split}
(\mathbf{S}_\mathrm{csub}(x))
\,
(\operatorname{adj}(\mathbf{S}_\mathrm{prin}(x,p)))
&=
(\mathbf{S}_\mathrm{prin}(x,A(x)))\,(\operatorname{adj}(\mathbf{S}_\mathrm{prin}(x,p))\\
&=\left(- \mathbf{g}^{\mu\nu}(x)\,A_\mu(x)\,p_\nu\right) \mathrm{Id}.
\end{split}
\end{equation}
Differentiating both sides of \eqref{lemma A temp 1} with respect to $p_\beta$, taking the matrix trace and lowering the index with the
($(-2)$-density valued) metric  yields \eqref{explicit formula for A}.
\end{proof}

Formulae
\eqref{definition of metric density},
\eqref{covariant subprincipal symbol}
and
\eqref{definition of electromagnetic covector potential}
tell us that the full symbol is completely
determined by principal symbol and electromagnetic covector potential.

\begin{lemma}
\label{lemma transformation of A}
Let $A$ be the electromagnetic covector potential
encoded within the sesquilinear form $S$.
Then the electromagnetic covector potential
$\tilde A$ encoded within the sesquilinear form $\tilde S$
defined by \eqref{sesquilinear form with tilde equation 1} is
\begin{equation}\label{transformation of A}
\tilde A=A+\frac12\operatorname{grad}(\arg\det R).
\end{equation}
\end{lemma}

\begin{proof}
From
\cite[formulae (5.1), (5.2), (D.4)--(D.6)]{nongeometric}
it follows that 
\begin{equation}\label{lemma covector potental proof 1}
\tilde{\mathbf{S}}_\mathrm{csub}=
R^*\mathbf{S}_\mathrm{csub}R-\mathbf{Q}-\mathbf{Q}^*,
\end{equation}
with
\begin{equation}\label{lemma covector potental proof 2}
\mathbf{Q}=-\frac{i}{8}\,\mathbf{g}_{\alpha\beta} \,
R^*\,(\mathbf{S}_\mathrm{prin})_{p_\alpha}\,
R_{x^\gamma}\,R^{-1}\,
(\operatorname{adj}\mathbf{S}_\mathrm{prin})_{p_\beta}\, (\mathbf{S}_\mathrm{prin})_{p^\gamma}\,R.
\end{equation}
The matrix-function $R$ can be written locally as
\begin{equation}
\label{local representaion of R GL}
R(x)=r(x)\,e^{i\varphi(x)}\,R_1(x),
\end{equation}
where $r,\varphi:M\to \mathbb{R}$ and $R_1:M\to SL(2,\mathbb{C})$
are smooth real and matrix-valued functions respectively. In particular, $\varphi(x)=\frac12 \arg \det R(x)$. From \eqref{local representaion of R GL} we obtain
\begin{equation}\label{lemma covector potental proof 3}
R_{x^\gamma}\,R^{-1}=(R_1)_{x^\gamma}\,R_1^{-1}+\dfrac{r_{x^\gamma}}{r}\operatorname{Id}+i\,\varphi_{x^\gamma} \operatorname{Id}.
\end{equation}
The first term on the RHS of \eqref{lemma covector potental proof 3} is trace-free and hence,
by \cite[formula (C.1)]{nongeometric},
it does not contribute to \eqref{lemma covector potental proof 2}. The second term is real, and, when multiplied by $i$, it does not contribute to $ \mathbf{Q}+\mathbf{Q}^*$. Therefore, by substituting \eqref{lemma covector potental proof 3} into \eqref{lemma covector potental proof 2} and, in turn, \eqref{lemma covector potental proof 2} into \eqref{lemma covector potental proof 1}, we obtain
\begin{equation}
\tilde{\mathbf{S}}_\mathrm{csub}=R^*\mathbf{S}_\mathrm{csub}R+ R^*(\mathbf{S}_\mathrm{prin})_{p^\gamma}\,\varphi_{x^\gamma} R,
\end{equation}
from which \eqref{transformation of A} ensues.
\end{proof}

\begin{remark}
The use of the term `electromagnetic covector potential' for
the covector field $A$ is motivated by
the fact that this $A$ is, in our context, a counterpart of what in
gauge theory is a $U(1)$-connection,
see formula \eqref{transformation of A}.
\end{remark}

\subsection{Topological objects}
\label{Topological objects}

As explained in the beginning of the previous subsection,
our manifold $M$ is \emph{a priori} parallelizable, hence orientable.
We specify an orientation on our manifold
and define the topological charge of our sesquilinear form as
\begin{equation}
\label{definition of relative orientation}
c_\mathrm{top}:=
-\frac i2\sqrt{-\det\mathbf{g}_{\alpha\beta}}\,\operatorname{tr}
\bigl(
(\mathbf{S}_\mathrm{prin})_{p_1}
(\mathbf{S}_\mathrm{prin})_{p_2}
(\mathbf{S}_\mathrm{prin})_{p_3}
(\mathbf{S}_\mathrm{prin})_{p_4}
\bigr),
\end{equation}
where $\operatorname{tr}$ stands for the matrix trace.
Straightforward calculations show that
the number $c_\mathrm{top}$
can take only two values, $+1$ or $-1$.
It describes the orientation of the principal symbol
relative to our chosen orientation of local coordinates $x=(x^1,x^2,x^3,x^4)$.

Our Lorentzian 4-manifold $(M,g)$ does, in fact, possess an additional property:
it is automatically time-orientable, i.e.~it admits a timelike (co)vector field.
Indeed, consider the quantity
\[
f_x(p):=\frac1{\rho(x)}\operatorname{tr}\mathbf{S}_\mathrm{prin}(x,p).
\]
We are looking at a linear map
\[
f_x: T^*_x M \to \mathbb{R}, \qquad p\mapsto f_x(p),
\]
depending smoothly on $x\in M$.
Non-degeneracy of our principal symbol implies that
\[
\operatorname{range}f_x\ne\{0\},\qquad\forall x\in M.
\]
By duality the linear map $f_x$ can be represented in terms of a nonvanishing
vector field $t$,
\[
f_x(p)=t(p)=t^\alpha(x)\,p_\alpha\,,
\]
which can be shown to be timelike.

Let us specify a time orientation by choosing
a reference timelike covector field $q$.
We define the temporal charge of our sesquilinear form as
\begin{equation}
\label{definition of temporal charge}
c_\mathrm{tem}:=
\operatorname{sgn}t(q).
\end{equation}
It describes the orientation of the principal symbol
relative to our chosen time orientation.

\begin{definition}
\label{definition of 2-torsion spin c structure}
Consider symbols corresponding to metrics from a given conformal class
and with the same topological and temporal charges.
We define \emph{2-torsion $\spin^c$ structure} to be
the equivalence class of symbols
\begin{equation}
\label{GL equivalence of principal symbols}
\left[\mathbf{S}\right]:=\{\tilde{\mathbf{S}}\ |\ \tilde{\mathbf{S}}_\mathrm{prin}=R^*\mathbf{S}_\mathrm{prin}R,\ R\in C^\infty(M,GL(n,\mathbb{C})) \}\,.
\end{equation}
\end{definition}

\begin{definition}
\label{definition of spin structure}
Consider symbols corresponding to a given metric
and with the same topological and temporal charges.
We define \emph{spin structure} to be
the equivalence class of symbols
\begin{equation}
\label{SL equivalence of principal symbols}
\left[\mathbf{S}\right]:=\{\tilde{\mathbf{S}}\ |\ \tilde{\mathbf{S}}_\mathrm{prin}=R^*\mathbf{S}_\mathrm{prin}R,\ R\in C^\infty(M,SL(n,\mathbb{C})) \}\,.
\end{equation}
\end{definition}

In the above definitions we use topological terminology,
even though the definitions themselves are stated in a purely
analytic fashion.
A rigorous justification for
this is provided in the next section.

\section{Transition from analysis to topology}
\label{Transition from analysis to topology}

The aim of this section is to perform an analysis of
Definitions
\ref{definition of 2-torsion spin c structure}
and
\ref{definition of spin structure},
so as to show that these
analytic definitions are equivalent to standard topological ones.
We will establish this equivalence by rewriting
the principal symbol of a sesquilinear form in a way that
is better suited for revealing topological content, see formula
\eqref{principal symbol via frame 2}
below.

\subsection{Framings and their equivalence}
Let $M$ be an oriented time-oriented Lorentzian $4$-manifold. By a \emph{frame} at a point $x \in M$ we mean a positively oriented and positively time-oriented orthonormal, in the Lorentzian sense, frame $e_j$, $j=1,2,3,4$, in the tangent space $T_x M$:
\[
\det e_j{}^\alpha>0,
\qquad
q(e^4)>0,
\]
\begin{equation*}
\label{Lorentzian orthonormality}
g_{\alpha\beta}\,e_j{}^\alpha e_k{}^\beta
=
\begin{cases}
\phantom{+}0\quad\text{if}\quad j\ne k,\\
+1\quad\text{if}\quad j=k\ne4,
\\
-1\quad\text{if}\quad j=k=4.
\end{cases}
\end{equation*}
Here each vector $e_j$ has coordinate components
$e_j{}^\alpha$, $\alpha=1,2,3,4$.
By a \emph{framing} of $M$ we mean a choice of frame at every point $x \in M$ depending smoothly on the point.
Of course, the contravariant metric tensor is expressed via the framing as
\begin{equation*}
\label{Lorentzian metric via framing}
g^{\alpha\beta}
=e_1{}^\alpha e_1{}^\beta
+e_2{}^\alpha e_2{}^\beta
+e_3{}^\alpha e_3{}^\beta
-e_4{}^\alpha e_4{}^\beta
\end{equation*}
and the Lorentzian density is expressed via the framing as
$\rho=(\det e_j{}^\alpha)^{-1}$.

Let 
\begin{equation*}
\label{standard basis}
s^1=s_1=
\begin{pmatrix}
0&1\\
1&0
\end{pmatrix},
\ \ s^2=s_2=
\begin{pmatrix}
0&-i\\
i&0
\end{pmatrix},
\ \ s^3=s_3=
\begin{pmatrix}
1&0\\
0&-1
\end{pmatrix},
\ \ s^4=-s_4=
\begin{pmatrix}
1&0\\
0&1
\end{pmatrix}
\end{equation*}
be the standard basis in the real vector space of $2\times 2$ Hermitian matrices. Then the principal symbols of sesquilinear forms
with $c_\mathrm{top}=c_\mathrm{temp}=+1$
are in one-to-one correspondence with framings. This correspondence is realised explicitly by the formula
\begin{equation}
\label{principal symbol via frame 2}
\mathbf{S}_\mathrm{prin}(x,p)=\rho(x)\,s^j\, e_j{}^\alpha(x)\,p_\alpha\,.
\end{equation}
The nondegeneracy condition \eqref{definition of non-degeneracy equation} implies that the vector fields $e_j$, $j=1,2,3,4$, are linearly independent. Moreover, they are automatically Lorentz-orthogonal
with respect to the metric encoded within $\mathbf{S}_\mathrm{prin}$,
see \cite[Sections 1 and 2]{jmp2017}.
Thus, the $e_j$, $j=1,2,3,4$, provide a framing. Observe that one can also argue the other way around:
in view of \eqref{principal symbol via frame 2}
a framing completely determines the principal symbol.

The point of the above argument is that instead of working with
an analytic object, a principal symbol, we can work with an
equivalent geometric
object, a framing.

In what follows $SO^+(3,1)$ denotes the identity component of the Lorentz group
and
\linebreak
$CSO^+(3,1)$ denotes its conformal extension. Here `conformal extension' refers to
multi\-plication of matrices from $SO^+(3,1)$ by arbitrary positive factors.
The Lie group $SO^+(3,1)$ is 6-dimensional,
so $CSO^+(3,1)$ is 7-dimensional.
The conformal extension of the Lorentz group is needed because gauge transformations
\eqref{sesquilinear form with tilde equation 1},
\eqref{matrix-function GL}
result in the scaling of the Lorentzian metric encoded within the  principal symbol,
see formula \eqref{conformal scaling of metric under gauge transformations}.

Let us now fix a conformal class of Lorentzian metrics and within this class
choose a pair of principal symbols
$\mathbf{S}_\mathrm{prin}$
and
$\tilde{\mathbf{S}}_\mathrm{prin}$.
Let $e_j$ and $\tilde e_j$ be the corresponding framings.
Then
\begin{equation}
\label{relation between framings}
\tilde e_j=O_j{}^k\,e_k
\end{equation}
for some uniquely defined smooth matrix-function $O:M\to CSO^+(3,1)$.

Suppose now
that there exists a matrix-function $R:M\to GL(2,\mathbb{C})$
such that
$\tilde{\mathbf{S}}_\mathrm{prin}=
R^*\,\mathbf{S}_\mathrm{prin}R$.
A straightforward calculation shows that the
matrix-function $O$
appearing in \eqref{relation between framings}
is expressed via $R$ as
\begin{equation}
\label{spin homomorphism formula 1}
O_j{}^k=\frac12|\det R|^{-4/3}\operatorname{tr}(s_jR^*s^kR).
\end{equation}

It is convenient to define
\begin{equation}
\label{new R 1}
\mathcal{R}:=|\det R|^{2/3}\,R.
\end{equation}
Of course, the above formula can be inverted:
\begin{equation}
\label{new R 2}
R=|\det\mathcal{R}|^{-2/7}\,\mathcal{R}.
\end{equation}
The advantage of working with the matrix-function
\begin{equation*}
\label{new matrix-function GL}
\mathcal{R}:M\to GL(2,\mathbb{C})
\end{equation*}
rather than the original matrix-function
\eqref{matrix-function GL}
is that formula
\eqref{spin homomorphism formula 1}
simplifies and reads now
\begin{equation}
\label{spin homomorphism formula 2}
O_j{}^k=\frac12\operatorname{tr}(s_j\mathcal{R}^*s^k\mathcal{R}).
\end{equation}
The switch from $R$ to $\mathcal{R}$ does not affect the topological issues
we are addressing, it just makes formulae simpler.

Observe that
when $\mathcal{R}\in SL(2,\mathbb C)$,
\eqref{spin homomorphism formula 2}
is the standard spin homomorphism formula
which provides a map
\begin{equation}
\label{spin homomorphism formula 3}
\Pi: SL(2,\mathbb C)\longrightarrow SO^+(3,1),\quad \Pi(\mathcal{R}) = O.
\end{equation}
When we allow $\mathcal{R}$ to take values in $GL(2,\mathbb C)$,
formula
\eqref{spin homomorphism formula 2}
gives us a map
\begin{equation}
\label{spin homomorphism formula 4}
\Pi: GL(2,\mathbb C)\longrightarrow CSO^+(3,1),\quad \Pi(\mathcal{R}) = O.
\end{equation}

We are now in a position to rephrase Definitions
\ref{definition of 2-torsion spin c structure}
and
\ref{definition of spin structure}
as follows.

Consider symbols corresponding to metrics from a given conformal class
and with the same topological and temporal charges.
We define \emph{2-torsion $\spin^c$ structure} to be
the equivalence class of symbols,
where two symbols are called equivalent
if the matrix-function $O$ relating them, see \eqref{relation between framings},
can be written in the form \eqref{spin homomorphism formula 2}
for some $\mathcal{R}:M\to GL(2,\mathbb C)$. In other words, the matrix-function
$O: M \longrightarrow CSO^+(3,1)$ admits a factorization
\begin{equation}\label{E:factor}
O: M \overset{\mathcal{R}}{\longrightarrow} GL(2,\mathbb C) \overset{\Pi}{\longrightarrow} CSO^+(3,1).
\end{equation}
If the metric is the same,
we define \emph{spin structure} to be
the equivalence class of symbols,
where two symbols are called equivalent
if the matrix-function $O$ relating them
can be written in the form \eqref{spin homomorphism formula 2}
for some $\mathcal{R}:M\to SL(2,\mathbb C)$.

\begin{remark}
\label{remark about uniqueness of R}
It is easy to see that
in the $GL$ case the matrix-function $\mathcal{R}$, if it exists,
is defined uniquely modulo multiplication by $e^{i\varphi}$,
where $\varphi$ is an arbitrary smooth real-valued scalar function.
In the $SL$ case the matrix-function $\mathcal{R}$, if it exists,
is defined uniquely modulo multiplication by $\pm1$.
\end{remark}

It follows from \cite{jmp2017}
that our definition of spin structure agrees with the accepted topological one\footnote{The map called $\Ad:SL(2,\mathbb{C})\to SO^+(3,1)$ in \cite{jmp2017} should in fact be understood as the spin homomorphism~$\Pi$.}.
In the remainder of this section we establish a similar result for
2-torsion $\spin^c$ structure.

Let us remind the reader that it follows from our assumptions
\eqref{definition of non-degeneracy equation}
and
\eqref{m equals 4 n equals 2}
that $M$ is a Lorentzian manifold which is parallelizable and time-orientable.
In particular, it is spin.
A choice of reference framing on $M$ provides a trivialization of the tangent bundle $TM$ so that any other framing is related to this reference framing by a smooth function $O: M \to CSO^+(3,1)$. Two framings corresponding to functions $O_1$ and $O_2$ are equivalent in the above sense if and only if there exists a smooth function $\mathcal{R}:M \to GL(2,\mathbb C)$ such that
$O_2\,\cdot\,(\Pi \circ\mathcal{R}) = O_1$ as functions $M \to CSO^+(3,1)$.




\subsection{Topological characterization}\label{Topological characterization}
In this section, we will characterize the equivalence relation we used to define the 2-torsion $\spin^c$ structures in purely topological terms. We begin by recalling that the compact subgroups $U(2) \subset GL(2,\mathbb C)$ and $SO(3) \subset CSO^+ (3,1)$ are deformation retracts of the respective non-compact Lie groups compatible with the map \eqref{spin homomorphism formula 4} in the sense that the following diagram commutes
\smallskip
\begin{equation}\label{diagram}
\begin{aligned}
\begin{tikzpicture}
\draw (8,6) node (a) {$U(2)$};
\draw (12,6) node (b) {$GL(2,\mathbb C)$};
\draw (8,3) node (c) {$SO(3)$};
\draw (12,3) node (d) {$CSO^+(3,1)$};
\draw[->](a)--(b);
\draw[->](c)--(d);
\draw[->](a)--(c) node [midway,right](TextNode){$\Ad$};
\draw[->](b)--(d) node [midway,right](TextNode){$\Pi$};
\end{tikzpicture}
\end{aligned}
\end{equation}
Here we used the fact that the restriction of the map \eqref{spin homomorphism formula 4} to the subgroup $U(2)$ coincides with the adjoint map $\Ad: U(2) \longrightarrow SO(3)$. The two vertical arrows in this diagram are principal $U(1)$--bundles, the action being multiplication by a diagonal matrix; see Remark \ref{remark about uniqueness of R}.

\begin{lemma}\label{L:bundles}
The principal $U(1)$--bundles $U(2) \to SO(3)$ and $GL(2,\mathbb C) \to CSO^+ (3,1)$ are non-trivial.
\end{lemma}

\begin{proof}
A principal bundle is known to be trivial if and only if it admits a section. Assuming that the bundle $U(2) \to SO(3)$ admits a section $s: SO(3) \to U(2)$, we immediately obtain a contradiction because the composition
\[
\begin{CD}
H^2 (SO(3);\mathbb Z) @> \Ad^* >>  H^2 (U(2);\mathbb Z) @> s^* >> H^2 (SO(3);\mathbb Z) 
\end{CD}
\]

\smallskip\noindent
must be identity while $H^2 (U(2);\mathbb Z) = \mathbb Z$ and $H^2 (SO(3);\mathbb Z) = \mathbb Z/2$. The argument for the other bundle is similar.
\end{proof}

We will be mostly interested in the bundle $GL(2,\mathbb C) \to CSO^+ (3,1)$. Given a map $f: M \to CSO^+(3,1)$, associate with it the cohomology class $\mathcal O (f) = f^*(1) \in H^2 (M;\mathbb Z)$, where $1 \in H^2 (CSO^+ (3,1); \mathbb Z) = \mathbb Z/2$ is the generator. Note that $\mathcal O(f)$ is an element of order at most two in $H^2 (M;\mathbb Z)$; in particular, it automatically vanishes whenever the group $H^2 (M;\mathbb Z)$ has no 2-torsion.

\begin{proposition}
A map $f: M \to CSO^+ (3,1)$ admits a factorization \eqref{E:factor} if and only if $\mathcal O (f) = 0$.
\end{proposition}

\begin{proof}
We begin by constructing, for a given map $f: M \to CSO^+ (3,1)$, the pull back principal bundle 
\smallskip
\begin {equation*}
\begin{tikzpicture}
\draw (8,6) node (a) {$E(f)$};
\draw (12,6) node (b) {$GL(2,\mathbb C)$};
\draw (8,3) node (c) {$M$};
\draw (12,3) node (d) {$CSO^+(3,1)$};
\draw[->](a)--(b);
\draw[->](c)--(d) node [midway, above](TextNode){$f$};
\draw[->](a)--(c) node [midway,right](TextNode){$\pi$};
\draw[->](b)--(d) node [midway,right](TextNode){$\Pi$};
\end{tikzpicture}
\end {equation*}
where $E(f) = \{\,(x,p)\, |\, f (x) = \Pi (p)\,\}\subset M \times GL(2,\mathbb C)$ and the maps $\pi: E(f) \to M$ and $E(f) \to GL(2,\mathbb C)$ are projections onto the respective factors. It is well known (and can be checked by comparing the definitions) that $f: M \to CSO^+ (3,1)$ admits a factorization \eqref{E:factor} if and only if the bundle $\pi: E(f) \to M$ admits a section. Since $\pi: E(f) \to M$ is a principal bundle it admits a section if and only if it is trivial. The latter happens if and only if the first Chern class $c_1 (E(f)) \in H^2 (M; \mathbb Z)$ vanishes. Since $c_1$ is natural with respect to pull backs, $c_1(E(f))$ is the pul back via $f^*: H^2 (CSO^+(3,1);\mathbb Z) \to H^2 (M;\mathbb Z)$ of the first Chern class of the bundle $GL(2,\mathbb C) \to CSO^+(3,1)$. According to Lemma \ref{L:bundles}, the latter bundle is non-trivial, hence its first Chern class must be the generator $1 \in H^2 (CSO^+ (3,1);\mathbb Z) = \mathbb Z/2$ and $c_1 (E(f)) = f^*(1) = \mathcal O(f)$.
\end{proof}

\begin{proposition}
Every element of $H^2 (M;\mathbb Z)$ of order two can be realized as $\mathcal O(f)$ for some map $f: M \to CSO^+ (3,1)$.
\end{proposition}

\begin{proof}
Let us consider the short exact sequence $0 \longrightarrow \mathbb Z \overset{\cdot 2}\longrightarrow \mathbb Z \longrightarrow \mathbb Z/2 \longrightarrow 0$. The associated long exact sequence
\[
\begin{CD}
\ldots @>>> H^1 (M;\mathbb Z/2) @> \partial >> H^2 (M;\mathbb Z) @> \cdot 2 >> H^2 (M;\mathbb Z) @>>> H^2 (M;\mathbb Z/2) @>>> \ldots \\
\end{CD}
\]

\medskip\noindent
implies that every element $b \in H^2 (M;\mathbb Z)$ of order two belongs to the image of the Bockstein homomorphism $\partial: H^1 (M;\mathbb Z/2) \longrightarrow H^2 (M;\mathbb Z)$. We will show that every cohomology class $a \in H^1(M; \mathbb Z/2)$ is of the form $a = f^*(1)$ for some $f: M \to SO^+(3,1)$ and $1 \in H^1(SO^+(3,1);\mathbb Z/2) = \mathbb Z/2$. The result will then follow from the commutative diagram 
\begin{equation}\label{E:bockstein}
\begin{CD}
H^1 (SO^+(3,1);\mathbb Z/2) @> \partial >> H^2 (SO^+(3,1);\mathbb Z) \\
@VV f^* V @VV f^* V \\
H^1 (M;\mathbb Z/2) @> \partial >> H^2 (M;\mathbb Z) \\
\end{CD}
\end{equation}
whose upper row is an isomorphism, and the fact that $SO^+(3,1) \subset CSO^+(3,1)$ is a deformation retract. 

Let us consider the double covering $SL(2,\mathbb C) \to SO^+(3,1)$ given by the spin homomorphism \eqref{spin homomorphism formula 3} and its associated fibration sequence (see, for instance, \cite[Lemma 8.23]{kirk_davis})
\[
\begin{CD}
\mathbb Z/2 @>>> SL(2,\mathbb C) @>>> SO^+(3,1) @>>> K(\mathbb Z/2,1) @>>> BSL(2,\mathbb C),
\end{CD}
\]
where $K(\mathbb Z/2,1)$ is the Eilenberg--MacLane space and $BSL(2,\mathbb C)$ the classifying space of the Lie group $SL(2,\mathbb C)$. It gives rise to the exact sequence of homotopy sets (see \cite[Theorem 6.29]{kirk_davis})
\[
\begin{CD}
[M,SO^+(3,1)] @>>> H^1(M;\mathbb Z/2) @>>> [M,BSL(2,\mathbb C)]
\end{CD}
\]
using the fact that $H^1(M;\mathbb Z/2) = [M,K(\mathbb Z/2,1)]$. We wish to show that the first map in this sequence is surjective or, equivalently, that the second map is zero. Write $H^1(M;\mathbb Z/2) = [M,\RP^\infty]$ using the homotopy equivalence between $K(\mathbb Z/2,1)$ and the real projective space $\RP^{\infty}$. Also observe that, up to homotopy equivalence, $BSL(2,\mathbb C) = BSU(2) = \HP^{\infty}$, the quaternionic projective space. Then the question becomes whether, for any continuous map $M \to \RP^{\infty}$, the composition $M \to \RP^{\infty} \to \HP^{\infty}$ with the natural inclusion $\RP^{\infty} \to \HP^{\infty}$ is homotopic to zero. Since $\dim M = 4$ and the 5--skeleton of the CW--complex $\HP^{\infty}$ is $\HP^1 = S^4$, the cellular approximation theorem reduces this question to an identical question about the composition $M \to \RP^4 \to S^4$. By the Hopf theorem, a map $M \to S^4$ is homotopic to zero if and only if the induced map $H^4 (S^4;\mathbb Z) \to H^4 (M;\mathbb Z)$ is zero. In our case, this last map splits as the composition 
\[
\begin{CD}
H^4 (S^4;\mathbb Z) @>>> H^4 (\RP^4;\mathbb Z) @>>> H^4 (M;\mathbb Z),
\end{CD}
\]
with $H^4 (\RP^4;\mathbb Z) = \mathbb Z/2$. Since $M$ is orientable, $H^4 (M;\mathbb Z)$ is a free abelian group, hence the second map in this composition must vanish. 

It is worth mentioning that the orientability of $M$ in this argument is essential: in general, realizability of cohomology classes in $H^2 (M;\mathbb Z)$ can be obstructed by the non-trivial quadruple cup-product on $H^1 (M;\mathbb Z/2)$.
\end{proof}

\begin{corollary}\label{C:torsion}
The set of 2-torsion $\spin^c$ structures on $M$ is in a bijective correspondence with the 2-torsion subgroup of $H^2 (M; \mathbb Z)$.
\end{corollary}


\subsection{Differential geometric characterization}\label{spinc}
Our goal in this subsection is to identify the equivalence classes of framings with the 2-torsion $\spin^c$ structures on $M$, whose definition is modelled after that in Riemannian geometry \cite{LM}; see Remark \ref{2-torsion} below. In the special case at hand, when the tangent bundle $TM$ is trivialized via the reference frame, it reads as follows. A 2-torsion $\spin^c$ structure on $M$ is an equivalence class of commutative diagrams
\smallskip
\begin {equation*}
\begin{tikzpicture}
\draw (8,6) node (a) {$M \times GL(2,\mathbb C)$};
\draw (8,3) node (b) {$M \times CSO^+(3,1)$};
\draw (11,4.5) node (c) {$M$};
\draw[->](a)--(b) node [midway,right](TextNode){$\Phi$};
\draw[->](a)--(c) node [midway,above](TextNode){$\pi$};
\draw[->](b)--(c) node [midway,above](TextNode){$\pi$};
\end{tikzpicture}
\end {equation*}
where $\pi$ stands for the projection onto the first factor, and the map $\Phi$ is equivariant in that $\Phi(x,g) = \Phi (x,1)\cdot \Pi(g)$ for all $x \in M$ and $g \in GL(2,\mathbb C)$. Two diagrams as above with the vertical maps $\Phi_1$ and $\Phi_2$ are called equivalent if there is a commutative diagram
\begin {equation*}
\begin{tikzpicture}
\draw (8,6) node (a) {$M \times GL(2,\mathbb C)$};
\draw (12,6) node (b) {$M \times GL(2,\mathbb C)$};
\draw (10,3.5) node (c) {$M \times CSO^+(3,1)$};
\draw[->](a)--(b) node [midway,above](TextNode){$A$};
\draw[->](a)--(c) node [midway,left](TextNode){$\Phi_1\;$};
\draw[->](b)--(c) node [midway,right](TextNode){$\;\;\Phi_2$};
\end{tikzpicture}
\end {equation*}
such that $\pi\circ A = \pi$ and the map $A$ is equivariant in that $A(x,g) = A(x,1)\cdot g$ for all $x \in M$ and $g \in GL(2,\mathbb C)$.

\begin{theorem}
\label{Theorem Lorentzian}
For parallelizable time-orientable Lorentzian 4-manifolds, the equivalence classes of framings as above are in bijective correspondence with the 2-torsion $\spin^c$ structures.
\end{theorem}

\begin{proof}
Using the commutativity of the first diagram, write $\Phi (x,g) = (x,\phi(x,g))$ for some function $\phi: M \times GL(2,\mathbb C) \to CSO^+(3,1)$ and observe that the equivariance condition on $\Phi$ translates into the equation $\phi(x,g) = \phi(x,1)\,\cdot\,\Pi(g)$. Therefore, the map $\Phi$ is uniquely determined by the map $\psi: M \to CSO^+(3,1)$ given by $\psi(x) = \phi(x,1)$. 

Similarly, write $A(x,g) = (x,\alpha(x,g))$ and observe that the equivariance condition on $A$ translates into the equation $\alpha(x,g) = \alpha(x,1)\,\cdot\,g$. Therefore, the map $A$ is uniquely determined by the map $\beta: M \to GL(2,\mathbb C)$ given by $\beta(x) = \alpha(x,1)$. One can easily check that the second commutative diagram then simply means that $\psi_2\,\cdot\,\Pi(\beta) = \psi_1$ as functions $M \to CSO^+(3,1)$.
\end{proof}

Theorem~\ref{Theorem Lorentzian} rigorously shows, in view of \eqref{principal symbol via frame 2}, the equivalence of two definitions of 2-torsion $\spin^c$ structure, the standard topological one and Definition~\ref{definition of 2-torsion spin c structure}. Note that the equivalence we established is not canonical in that it depends on the choice of reference frame.

\begin{remark}\label{2-torsion}
It may be worth explaining the origin of the term `2-torsion $\spin^c$ structure'. Following the analogy with Riemannian geometry, one can define a $\spin^c$ structure on $M$ as the equivalence class of lifts of the principal frame bundle of $M$ to a $GL(2,\mathbb C)$ bundle; even though the frame bundle of $M$ is trivial, it may lift to a non-trivial $GL(2,\mathbb C)$ bundle. Among these lifts is the lift to an $SL(2,\mathbb C)$ bundle $P$ associated with the spin structure on $M$. The bundle $P$ must be trivial for topological reasons: it is classified by its second Chern class $c_2 (P)$, and we know that $4c_2 (P) = -p_1 (TM) = 0 \in H^4 (M; \mathbb Z)$. As in the Riemannian case, one can use $P$ to establish a bijective correspondence between $\spin^c$ structures on $M$ and the group $H^2 (M; \mathbb Z)$. Under this correspondence, the $\spin^c$ structure corresponding to a cohomology class $a \in H^2 (M; \mathbb Z)$ lives in a Hermitian rank-two bundle with the first Chern class $c_1 (P) + 2 a = 2a \in H^2 (M; \mathbb Z)$. Since we restrict ourselves to trivial bundles, the class $2a$ must vanish. This means that $a \in H^2 (M;\mathbb Z)$ is a 2-torsion, hence the name of the corresponding $\spin^c$ structure. 

\end{remark}


\section{Proofs of main theorems}
\label{Proofs of main theorems}

\subsection{Proof of Theorem~\ref{main theorem GL}}
\label{Proof of main theorem GL}

\subsubsection*{Necessity}
Let us first show that conditions (i)--(v) of Theorem \ref{main theorem GL} are necessary.

\begin{enumerate}[(i)]
\item[{}](i)
Formula \eqref{conformal scaling of metric under gauge transformations}
tells us that the conformal class of metrics is preserved under $GL$ trans\-formations,
so condition (i) is necessary.
\item[{}](ii)
Lemma \ref{lemma transformation of A} tells us that condition (ii) is necessary.
\item[{}](iii)-(iv)
In order to deal with conditions (iii) and (iv) we observe that the two charges,
topological \eqref{definition of relative orientation}
and 
temporal \eqref{definition of temporal charge},
can be expressed via the framing as
\[
c_\mathrm{top}=\operatorname{sgn}\det e_j{}^\alpha,
\qquad
c_\mathrm{tem}=
\operatorname{sgn}q(e^4).
\]
We showed in Section~\ref{Transition from analysis to topology} that under
$GL$ transformations the framing stays within the original connected component
of the conformally extended Lorentz group, hence conditions (iii) and (iv) are necessary.
\item[{}](v)
As to the necessity of condition (v), it follows immediately from 
Definition~\ref{definition of 2-torsion spin c structure}.
\end{enumerate}

\subsubsection*{Sufficiency}
Let us now show that conditions (i)--(v) of Theorem \ref{main theorem GL} are sufficient.

We need to find a $GL$ transformation
which turns one full symbol into the other.
As explained in subsection \ref{Geometric objects},
a full symbol is completely
determined by principal symbol and electromagnetic covector potential.
Thus, we need to find a $GL$ transformation
which turns one principal symbol into the other
and one electromagnetic covector potential into the other.

Conditions (i) and (iii)--(v) ensure that we can find a matrix-function \eqref{matrix-function GL}
which turns one principal symbol into the other,
see formula \eqref{GL equivalence of principal symbols}.
Remark \ref{remark about uniqueness of R}
and formulae \eqref{new R 1}, \eqref{new R 2}
tell us that this matrix-function \eqref{matrix-function GL}
is defined uniquely modulo multiplication by $e^{i\varphi}$,
where $\varphi$ is an arbitrary smooth real-valued scalar function.
In view of condition (ii) this function $\varphi$
can be chosen so as to
turn one electromagnetic covector potential into the other.

All in all, we obtain a matrix-function $R$ defined uniquely
modulo multiplication by a constant $c\in\mathbb{C}$, $|c|=1$.
\qed

\subsection{Proof of Theorem~\ref{main theorem SL}}
\label{Proof of main theorem SL}
The proof of Theorem~\ref{main theorem SL}
is similar to that of Theorem~\ref{main theorem GL},
with only two modifications.
\begin{itemize}
\item
$SL$ transformations preserve the metric,
so the requirement is that the two metrics are the same
as opposed to the two metrics being in the same conformal class.
\item
$SL$ transformations preserve the electromagnetic covector potential,
so the requirement is that the two electromagnetic covector potentials are the same
as opposed to the two electromagnetic covector potentials being in the same  cohomology class in $H^1_{\mathrm{dR}}(M)$.

All in all, we obtain a matrix-function $R$ defined uniquely
modulo multiplication by $\pm1$.
\qed
\end{itemize}

\section{The 3-dimensional Riemannian case}
\label{The 3-dimensional Riemannian case}

Let us consider first order sesquilinear forms satisfying the additional assumption
\begin{equation}
\label{principal symbol is trace-free}
\operatorname{tr}\mathbf{S}_\mathrm{prin}(x,p)=0,\qquad\forall(x,p)\in T^*M.
\end{equation}
In this setting it is natural to look at transformations of symbols generated by matrix-functions
\begin{equation}
\label{matrix-function U}
R:M\to U(n)
\end{equation}
or
\begin{equation}
\label{matrix-function SU}
R:M\to SU(n).
\end{equation}
Of course,
$U(n)\subset  GL(n,\mathbb{C})$
and 
$SU(n)\subset  SL(n,\mathbb{C})$,
so
\eqref{matrix-function U}
and
\eqref{matrix-function SU}
are special cases of
\eqref{matrix-function GL}
and
\eqref{matrix-function SL}
respectively.
We are now more restrictive in our choice of matrix-functions $R$
because we want to preserve condition \eqref{principal symbol is trace-free}.

It turns out that for sesquilinear forms with
trace-free principal symbol one can perform
a classification similar to that described in previous
sections. We list the main results below,
skipping detailed proofs as these are modifications
of arguments presented earlier in the paper.

Condition \eqref{m equals n squared}
is now replaced by
\begin{equation}
\label{m equals n squared minus 1}
m=n^2-1.
\end{equation}
Under the condition \eqref{m equals n squared minus 1}
a manifold $M$ admits a non-degenerate Hermitian first order sesqui\-linear form
with trace-free principal symbol
if and only if it is parallelizable.
So further on we assume that our manifold is parallelizable.

In this section we deal with the special case
\begin{equation}
\label{m equals 3 n equals 2}
m=3,\qquad n=2,
\end{equation}
compare with \eqref{m equals 4 n equals 2}.
It is known \cite{Stiefel,Kirby}
that a 3-manifold is parallelizable if and only if it is orientable.
Therefore, orientability is our only topological restriction on $M$.

It is easy to see that under the assumption
\eqref{principal symbol is trace-free}
the non-degeneracy condition
\eqref{definition of non-degeneracy equation}
is equivalent to the condition
\begin{equation}
\label{definition of ellipticity}
\det\mathbf{S}_\mathrm{prin}(x,p)\ne0,\qquad\forall(x,p)\in T^*M\setminus\{0\}.
\end{equation}
But \eqref{definition of ellipticity} is the standard ellipticity condition.
Thus, in this section we work with
formally self-adjoint elliptic
first order sesquilinear forms $S$ with trace-free principal symbols
which act on sections of the trivial
$\mathbb{C}^2$-bundle over a connected smooth oriented 3-manifold $M$ without boundary.

We define $\mathbf{g}^{\alpha\beta}(x)$ via \eqref{definition of metric density}.
It is easy to see that the quadratic form
$\mathbf{g}^{\alpha\beta}$
is positive definite.
This implies, in particular, that
\begin{equation*}
\label{metric density is non-degenerate dimension 3}
\det\mathbf{g}^{\alpha\beta}(x)>0,
\qquad \forall x\in M.
\end{equation*}
Put
\begin{equation}
\label{initial definition of Riemannian density}
\rho(x):=(\det\mathbf{g}^{\mu\nu}(x))^{1/4}.
\end{equation}
The quantity \eqref{initial definition of Riemannian density}
is a density.
This observation allows us to define the Riemannian metric
$g^{\alpha\beta}(x):=
(\rho(x))^{-2}\ \mathbf{g}^{\alpha\beta}(x)$.
Of course, formula
\eqref{initial definition of Riemannian density}
can now be rewritten in more familiar form as
$\rho(x)=(\det g_{\mu\nu}(x))^{1/2}\,$.
And it is easy to see that our metric tensor
is invariant under transformations
\eqref{sesquilinear form with tilde equation 1},
\eqref{matrix-function U}.

We define the covariant subprincipal symbol in accordance with formula
\eqref{covariant subprincipal symbol}.
The magnetic covector potential $A=(A_1,A_2,A_3)$
and electric potential $A_4$ are defined as the solution of
\begin{equation*}
\label{definition of magnetic covector potential}
\mathbf{S}_\mathrm{csub}(x)=\mathbf{S}_\mathrm{prin}(x,A(x))+A_4\,\mathrm{Id}\,,
\end{equation*}
compare with \eqref{definition of electromagnetic covector potential}.
For the magnetic potential we still have the explicit formula
\eqref{explicit formula for A}
and for the electric potential we have
\begin{equation*}
\label{explicit formula for A4}
A_4
=\frac12\,
\operatorname{tr}
\mathbf{S}_\mathrm{csub}\,.
\end{equation*}

The full symbol
is completely
determined by principal symbol, magnetic covector potential and electric potential.
The electric potential is invariant under transformations
\eqref{sesquilinear form with tilde equation 1},
\eqref{matrix-function U},
whereas the magnetic covector potential transforms in accordance with formula
\eqref{transformation of A}.

We specify an orientation on our manifold
and define the topological charge of our sesqui\-linear form as
\begin{equation}
\label{definition of relative orientation dimension 3}
c_\mathrm{top}:=
-\frac i2\sqrt{\det\mathbf{g}_{\alpha\beta}}\,\operatorname{tr}
\bigl(
(\mathbf{S}_\mathrm{prin})_{p_1}
(\mathbf{S}_\mathrm{prin})_{p_2}
(\mathbf{S}_\mathrm{prin})_{p_3}
\bigr)
=
\operatorname{sgn}\det e_j{}^\alpha,
\end{equation}
compare with \eqref{definition of relative orientation}.

\begin{definition}
\label{definition of spin c structure dimension 3}
Consider symbols corresponding to a given metric
and with the same topological charge.
We define \emph{2-torsion $\spin^c$ structure} 
to be the equivalence class of symbols
\begin{equation}
\label{U equivalence of principal symbols}
\left[\mathbf{S}\right]=\{\tilde{\mathbf{S}}\ |\ \tilde{\mathbf{S}}_\mathrm{prin}=R^*\mathbf{S}_\mathrm{prin}R,\ R\in C^\infty(M,U(2)) \}\,.
\end{equation}
\end{definition}

\begin{definition}
\label{definition of spin structure dimension 3}
Consider symbols corresponding to a given metric
and with the same topological charge.
We define \emph{spin structure} to be the equivalence class of symbols
\begin{equation}
\label{SU equivalence of principal symbols}
\left[\mathbf{S}\right]=\{\tilde{\mathbf{S}}\ |\ \tilde{\mathbf{S}}_\mathrm{prin}=R^*\mathbf{S}_\mathrm{prin}R,\ R\in C^\infty(M,SU(2)) \}\,.
\end{equation}
\end{definition}

Our analytic definition of 2-torsion $\spin^c$ structure in dimension three, Definition~\ref{definition of spin c structure dimension 3}, is equivalent to the standard topological one. This follows by the argument of Section~\ref{Topological characterization} and Section~\ref{spinc} once the map $GL(2,\mathbb C) \to CSO^+(3,1)$ is replaced by the map $U(2) \to SO(3)$. Our analytic definition of spin structure in dimension three, Definition~\ref{definition of spin structure dimension 3}, is also equivalent to the standard topological one, which follows from
\cite{jmp2017} with the help of Diagram \ref{diagram}. 

We define $U$-equivalence and $SU$-equivalence
of symbols as in Definition~\ref{definition of GL equivalent sesquilinear forms},
replacing
\eqref{matrix-function GL}
by
\eqref{matrix-function U}
and
\eqref{matrix-function SU}
respectively.

We have the following analogues of
Theorems \ref{main theorem GL} and \ref{main theorem SL}.

\begin{theorem}
\label{main theorem U}
Two full symbols
$\mathbf{S}_\mathrm{full}(x,p)$ and $\,\tilde{\mathbf{S}}_\mathrm{full}(x,p)$
are $U$-\emph{equivalent} if and only if
\begin{enumerate}[(i)]
\itemsep0em
\item
the metrics encoded within these symbols are the same,
\item
the electric potentials encoded within these symbols are the same,
\item
the magnetic covector potentials encoded within these symbols
belong to the same cohomo\-logy class in $H^1_{\mathrm{dR}}(M)$,
\item
their topological charges are the same and
\item
they have the same 2-torsion $\spin^c$ structure.
\end{enumerate}
\end{theorem}

\begin{theorem}
\label{main theorem SU}
Two full symbols
$\mathbf{S}_\mathrm{full}(x,p)$ and $\,\tilde{\mathbf{S}}_\mathrm{full}(x,p)$
are $SU$-\emph{equivalent} if and only if
\begin{enumerate}[(i)]
\itemsep0em
\item
the metrics encoded within these symbols are the same,
\item
the electric potentials encoded within these symbols are the same,
\item
the magnetic covector potentials encoded within these symbols are the same,
\item
their topological charges are the same and
\item
they have the same spin structure.
\end{enumerate}
\end{theorem}

\subsection{Explicit examples}

Concluding this section, we examine two explicit examples.
The first one illustrates how topological obstructions may arise
when classifying symbols in accordance with \eqref{U equivalence of principal symbols}.
The second demonstrates the difference between spin and $\spin^c$.

\subsubsection{The Lie group $SO(3)$}

Let $M=SO(3)$. We claim that $SO(3)$ has more than one 2-torsion $\spin^c$ structure. This follows from Corollary \ref{C:torsion} and the non-vanishing of the group $H^2(SO(3);\mathbb Z)$ but can also be seen directly as follows. With reference to Section~\ref{Topological characterization}, consider the identity map
\[
\mathrm{Id}: M \to SO(3).
\]
The map $\mathrm{Id}$ does not lift to a map $SO(3)\to U(2)$, namely, there does not exist a map $s:SO(3)\to U(2)$ such that the diagram
\begin {equation*}
\begin{tikzpicture}
\draw (12,6) node (a) {$U(2)$};
\draw (8,3)  node (b) {$SO(3)$};
\draw (12,3) node (c) {$SO(3)$};
\draw[dashed,->](b)--(a) node [midway,above](TextNode){$s$};
\draw[->](a)--(c) node [midway,right](TextNode){$\pi$};
\draw[->](b)--(c) node [midway,above](TextNode){$\operatorname{Id}$};
\end{tikzpicture}
\end {equation*}
commutes. A cohomological argument can be found in the proof of Lemma~\ref{L:bundles}. Another way to see this is as follows. Let us restrict ourselves to $SU(2)$ matrices with zero trace. These matrices form a sphere $S^2 \subset SU(2)$, which can also be viewed as the conjugacy class of $\operatorname{diag}(i,-i)\in SU(2)$. Explicitly, the matrices in $S^2$ are of the form 
\[
A = \begin{pmatrix}
ia & b + ic \\ -b + ic & -ia
\end{pmatrix},
\]
where $a$, $b$, and $c$ are real numbers such that $a^2 + b^2 + c^2 = 1$. The adjoint representation sends matrices $A$ and $-A \in S^2$ to the same matrix, giving rise to the double covering $S^2 \to \mathbb R \rm P^2$ of the real projective plane. We shall show that the bundle $U(2) \to SO(3)$ does not admit a section even over the subset $\mathbb R \rm P^2 \subset SO(3)$. The issue one encounters with finding such a section is adjusting for the signs of $SU(2)$ matrices in $S^2$ mapping to the same matrix in $\mathbb R \rm P^2$.  To make this adjustment, we need to find a continuous function $h: S^2 \to U(1)$ such that $h(-x) = -h(x)$, where $-x$ stands for the antipodal map on the sphere. If such a function $h$ existed, its composition with the standard inclusion $U(1) \to \mathbb R^2$ would give rise to a function $f: S^2 \to \mathbb R^2$ with the property that $f(-x) = -f(x)$. However, such a function does not exist by the Borsuk--Ulam theorem \cite[Theorem~1.10]{hatcher}: the Borsuk--Ulam theorem states that, for any continuous function $f: S^2 \to \mathbb R^2$, there exists $x \in S^2$ such that $f(-x) = f(x)$. Combined with $f(-x) = -f(x)$, this means that $f(x) = 0$ for some $x$, which contradicts the fact that the image of $f$ belongs to the unit circle.

In fact, one can show that $SO(3)$ has precisely two distinct 2-torsion $\spin^c$ structures
and precisely two distinct spin structures because $H^2(SO(3); \mathbb Z) = \mathbb Z/2$ and $H^1 (SO(3); \mathbb Z/2) = \mathbb Z/2$. In this particular case, $\spin^c$ and spin structures are matched via the Bockstein isomorphism $H^1 (SO(3);\mathbb Z/2) \to H^2 (SO(3);\mathbb Z)$, cf.\ Diagram~\eqref{E:bockstein}.

\subsubsection{The 3-torus}

Let $M=\mathbb{T}^3$ be the 3-dimensional torus
parameterised by $\operatorname{mod}2\pi$ coordinates $x^\alpha$, $\alpha=1,2,3$.
Put
\[
\mathbf{S}(x,p)
=
\mathbf{S}_\mathrm{prin}(x,p)
:=
\begin{pmatrix}
p_3&p_1-ip_2\\
p_1+ip_2&-p_3
\end{pmatrix},
\]
\[
\tilde{\mathbf{S}}(x,p)
=
\tilde{\mathbf{S}}_\mathrm{prin}(x,p)
:=
\begin{pmatrix}
p_3&e^{ix^3}(p_1-ip_2)\\
e^{-ix^3}(p_1+ip_2)&-p_3
\end{pmatrix}.
\]
We have 
\[
\det\mathbf{S}_\mathrm{prin}(x,p)=
\det\tilde{\mathbf{S}}_\mathrm{prin}(x,p)=
-\left(
p_1^2+p_2^2+p_3^2
\right),
\]
which means that the metric encoded within
the symbols $\mathbf{S}$ and $\tilde{\mathbf{S}}$
is the same, namely, the Euclidean metric.
Furthermore, the topological charge
\eqref{definition of relative orientation dimension 3}
encoded within
the symbols $\mathbf{S}$ and $\tilde{\mathbf{S}}$
is the same, $+1$.
Do these symbols have the same $\spin^c$ structure?
The answer is yes, because if we take
\[
R(x)
=
\begin{pmatrix}
e^{-ix^3}&0\\
0&1
\end{pmatrix}
\in C^\infty(M,U(2)) 
\]
we get
\begin{equation}
\label{An example: the 3-torus relation between two principal symbols}
\tilde{\mathbf{S}}_\mathrm{prin}=R^*\mathbf{S}_\mathrm{prin}R.
\end{equation}
However, it is easy to see that there does not exist a matrix-function
$R\in C^\infty(M,SU(2))$ which would give \eqref{An example: the 3-torus relation between two principal symbols},
so our two symbols,
$\mathbf{S}$ and $\tilde{\mathbf{S}}$,
have different spin structure.

In fact, it follows from Corollary \ref{C:torsion} that the 3-torus has a unique 2-torsion $\spin^c$ structure because the cohomology group $H^2(\mathbb T^3; \mathbb Z) = \mathbb Z^3$ has no 2-torsion, but it has eight distinct spin structures because the cohomology group $H^1(\mathbb T^3; \mathbb Z/2) = (\mathbb Z/2)^3$ has eight elements.

\section{Sesquilinear forms vs linear operators}
\label{Sesquilinear forms vs linear operators}

Having developed our theory,
we are now in a position to connect
the motivational ideas outlined in the
Introduction with the theory of partial differential equations.

Consider an
Hermitian first order sequilinear form of the type
\eqref{definition of sequilinear 2}
on the infinite-dimensional vector space
$C_0^\infty(M,\mathbb{C}^2)$.

\subsection{Four-dimensional case}

%
In dimension $m=4$, introduce an inner product
\begin{equation}
\label{inner product on manifold}
\langle u,v\rangle
:=\int_M u^*Gv\,\rho\,dx\,,
\end{equation}
where $G$ is some positive definite Hermitian $2\times2$ matrix-function
and $\rho$ is the Lorentzian density defined as in
subsection \ref{Geometric objects}.
Our first order Hermitian sesquilinear form $S$ and inner product
\eqref{inner product on manifold} define a formally self-adjoint
first order linear differential operator $L$.
The problem here is that it is impossible to choose
$G$ so as to have
\[
R^*GR=G,\qquad\forall R\in GL(2,\mathbb{C}),
\]
or even
\[
R^*GR=G,\qquad\forall R\in SL(2,\mathbb{C}),
\]
i.e.~one cannot introduce
an inner product compatible with our gauge transformations.
Hence, in the 4-dimensional case the construction presented
in our paper defines a linear field equation
but not a linear operator.

\subsection{Three-dimensional case}

Working in dimension $m=3$ and within the framework of
Section~\ref{The 3-dimensional Riemannian case}
(see, in particular, formulae
\eqref{principal symbol is trace-free}--\eqref{matrix-function SU}),
introduce the inner product
\begin{equation}
\label{inner product on manifold dimension 3}
\langle u,v\rangle
:=\int_M u^*v\,\rho\,dx\,,
\end{equation}
where $\rho$ is the Riemannian density encoded within our sesquilinear form
in accordance with formulae 
\eqref{definition of metric density} and \eqref{initial definition of Riemannian density}.
Now \eqref{inner product on manifold dimension 3} is
compatible with our gauge transformations.
Hence, in the 3-dimensional case our construction
defines a formally self-adjoint elliptic first order linear
differential operator.

\section{Applications}
\label{Applications}

In dimension $m=4$ a distinguished physically meaningful
sesquilinear form is the so-called \emph{Weyl form}.
It is defined by the condition that the electromagnetic covector
potential is zero. The gauge group is
$SL(2,\mathbb{C})$. One cannot use here the gauge group
$GL(2,\mathbb{C})$ because the
electro\-magnetic covector
potential is not invariant under the action of this group,
see Lemma~\ref{lemma transformation of A}.
The corresponding linear field equation is called
\emph{Weyl's equation},
the accepted mathematical model for the massless neutrino
in curved spacetime.
The condition $A=0$ translates, in physical terms,
into the neutrino having no electric charge and, therefore,
not interacting with the electromagnetic field.

In dimension $m=3$ and under the assumption
\eqref{principal symbol is trace-free}
a distinguished physically meaningful
sesquilinear form is the so-called \emph{massless Dirac form}.
It is defined by the condition that the electric potential
and magnetic covector
potential are both zero.
By analogy with the previous paragraph,
the gauge group here is $SU(2)$ and one cannot use
$U(2)$ because the
magnetic covector
potential is not invariant under the action of the latter.
The corresponding linear differential operator is called
\emph{massless Dirac operator}.
Its spectrum describes the energy levels of a massless neutrino
in curved space.

\section{Acknowledgements}
The authors are grateful to Boris Botvinnik for generously sharing his expertise.

\begin{appendices}

\section{The concepts of principal and subprincipal symbol}
\label{The concepts of principal and subprincipal symbol}

The concepts of principal and subprincipal symbol
are widely used in modern analysis,
however they are traditionally employed for the description of
(pseudo)differential \emph{operators}.
In the main text of our paper
we use these concepts for the description of
\emph{sesquilinear forms}.
We explain below the relation between the two seemingly
different versions of, essentially, the same objects.

A half-density is a spatially varying complex-valued
quantity on $M$ which under changes of local
coordinates transforms as the square root of a density.
Analysts, especially those working in the spectral theory of partial differential operators,
often prefer working with half-densities rather than with scalar functions.

Let $L^{(1/2)}$ be a first order linear differential
operator acting on $n$-columns of half-densities.
In local coordinates this operator reads
\begin{equation}
\label{definition of operator on half-densities}
L^{(1/2)}=-iE^\alpha(x)\frac\partial{\partial x^\alpha}+F(x),
\end{equation}
where $E^\alpha(x)$ and $F(x)$ are some $n\times n$ matrix-functions,
compare with \eqref{definition of sequilinear 2}.
Here the superscript $(1/2)$ indicates that we are dealing with an
operator acting on  half-densities.

We define the principal, subprincipal and full symbols of the operator
\eqref{definition of operator on half-densities}
as
\begin{equation}
\label{definition of principal symbol}
L^{(1/2)}_\mathrm{prin}(x,p)
:=
E^\alpha(x)\,p_\alpha\,,
\end{equation}
\begin{equation}
\label{definition of subprincipal symbol}
L^{(1/2)}_\mathrm{sub}(x)
:=
F(x)+\frac i2\bigl(L^{(1/2)}_\mathrm{prin}\bigr)_{x^\alpha p_\alpha}(x)
=
F(x)+\frac i2(E^\alpha)_{x^\alpha}(x)\,,
\end{equation}
\begin{equation}
\label{definition of full symbol}
L^{(1/2)}_\mathrm{full}(x,p)
:=
L^{(1/2)}_\mathrm{prin}(x,p)
+
L^{(1/2)}_\mathrm{sub}(x)
\end{equation}
respectively.
It is easy to see that the full symbol $L^{(1/2)}_\mathrm{full}$
uniquely determines our first order linear differential operator $L^{(1/2)}$.

The definition of the subprincipal symbol \eqref{definition of subprincipal symbol}
originates from the classical paper \cite{DuiHor} of
J.J.~Duistermaat and L.~H\"ormander: see formula (5.2.8) in that paper.
Unlike \cite{DuiHor}, we work with matrix-valued symbols, but this
does not affect the formal definition of the subprincipal symbol.
The correction term $\frac i2\bigl(L^{(1/2)}_\mathrm{prin}\bigr)_{x^\alpha p_\alpha}$
plays a crucial role in formula \eqref{definition of subprincipal symbol}:
its presence ensures that the subprincipal symbol is invariant under changes of local coordinates.

Our formulae
\eqref{definition of density-valued principal symbol}--\eqref{definition of density-valued full symbol}
are analogues of the standard formulae
\eqref{definition of principal symbol}--\eqref{definition of full symbol}.
The bold script in the former indicates that we are dealing with density-valued quantities.

In order to establish the relation between symbols of sesquilinear forms and symbols of operators,
let us fix a particular positive density $\mu$ and introduce the inner product
\begin{equation}
\label{inner product on manifold with mu}
\langle u,v\rangle
:=\int_M u^*v\,\mu\,dx
\end{equation}
on $n$-columns of scalar fields.
Formulae
\eqref{relation between form and operator},
\eqref{definition of sequilinear 2}
and
\eqref{inner product on manifold with mu}
define a linear operator $L$.

The main result of this appendix is the following lemma.

\begin{lemma}
Conditions
\begin{equation}
\label{relation between operators}
L=\mu^{-1/2}L^{(1/2)}\mu^{1/2}
\end{equation}
and
\begin{equation}
\label{relation between symbols}
\mathbf{S}_\mathrm{full}=\mu L^{(1/2)}_\mathrm{full}
\end{equation}
are equivalent.
\end{lemma}

\begin{proof}
Formula \eqref{definition of operator on half-densities} implies
\begin{equation}
\label{a}
\mu^{-1/2}L^{(1/2)}\mu^{1/2}=-iE^\alpha\frac\partial{\partial x^\alpha}+F-\frac i2\,E^\alpha\,(\ln\mu)_{x^\alpha}\,.
\end{equation}
Performing integration by parts, we rewrite
formula \eqref{definition of sequilinear 2} as
\[
S(u,v)=\int_M
u^*
\left[
\frac1\mu
\left(
-i\mathbf{E}^\alpha\frac\partial{\partial x^\alpha}
+
\mathbf{F}
-
\frac i2(\mathbf{E}^\alpha)_{x^\alpha}
\right)
v
\right]\mu\,dx\,,
\]
which gives us the following explicit local representation of the operator $L\,$:
\begin{equation}
\label{b}
L
=
\frac1\mu
\left(
-i\mathbf{E}^\alpha\frac\partial{\partial x^\alpha}
+
\mathbf{F}
-
\frac i2(\mathbf{E}^\alpha)_{x^\alpha}
\right).
\end{equation}
Substituting \eqref{b} and \eqref{a} into \eqref{relation between operators},
we see that the latter reduces to the pair of equations
\begin{equation}
\label{c}
\mathbf{E}^\alpha=\mu E^\alpha,
\end{equation}
\begin{equation}
\label{d}
\mathbf{F}
-
\frac i2(\mathbf{E}^\alpha)_{x^\alpha}
=
\mu
\left(
F-\frac i2\,E^\alpha\,(\ln\mu)_{x^\alpha}
\right).
\end{equation}
Substituting \eqref{c} into \eqref{d}  we rewrite the latter in equivalent form
\begin{equation}
\label{e}
\mathbf{F}
=
\mu
\left(
F+\frac i2(E^\alpha)_{x^\alpha}
\right).
\end{equation}
In view of 
\eqref{definition of density-valued principal symbol}--\eqref{definition of density-valued full symbol}
and
\eqref{definition of principal symbol}--\eqref{definition of full symbol}
conditions \eqref{c} and \eqref{e} are equivalent to
\eqref{relation between symbols}.
\end{proof}

As already pointed out in Section~\ref{Sesquilinear forms vs linear operators},
in the most general setting of arbitrary $m$ (dimension of the manifold),
arbitrary $n$ (number of scalar fields)
and arbitrary sesquilinear form
the introduction of an inner product
of the form \eqref{inner product on manifold with mu} does not make much sense
because this inner product is incompatible with general linear and special linear
gauge transformations.
However, it makes sense in the special case
\eqref{m equals 3 n equals 2},
\eqref{principal symbol is trace-free}
because the inner product
\eqref{inner product on manifold with mu}
is compatible with unitary and special unitary
gauge transformations.
And in this special case it is natural to take $\mu=\rho\,$,
where $\rho$ is the Riemannian density encoded within our sesquilinear form
in accordance with formulae 
\eqref{definition of metric density} and \eqref{initial definition of Riemannian density}.

\end{appendices}


\begin{thebibliography}{8}

\bibitem{jmp2017}
Z.~Avetisyan, Y.-L.~Fang, N.~Saveliev and D.~Vassiliev,
Analytic definition of spin structure.
{\it J.~Math.~Phys.}~\textbf{58} (2017) 082301.

\bibitem{kirk_davis}
J.~Davis and P.~Kirk,
{\it Lecture Notes in Algebraic Topology} 
(American Mathematical Society, Providence, 2001).

\bibitem{DuiHor}
J.J.~Duistermaat and L.~H\"ormander,
Fourier integral operators II.
{\it Acta~Math.}~\textbf{128} (1972) 183--269.

\bibitem{nongeometric}
Y.-L.~Fang and D.~Vassiliev,
Analysis as a source of geometry: a non-geometric representation of the Dirac equation.
{\it J.~Phys.~A}~\textbf{48} (2015) 165203.

\bibitem{hatcher}
A.~Hatcher,
{\it Algebraic topology} 
(Cambridge University Press, Cambridge, 2002).

\bibitem{Kirby}
R.~C.~Kirby,
{\it The topology of 4-manifolds}
(Lecture Notes in Mathematics, 1374, Springer-Verlag, Berlin, 1989).

\bibitem{LM}
H.~B.~Lawson and M.-L.~Michelsohn,
\emph{Spin Geometry} (Princeton University Press, 1989).


\bibitem{Stiefel}
E.~Stiefel,
Richtungsfelder und Fernparallelismus in $n$-dimensionalen
Mannigfaltigkeiten.
{\it Comment.~Math.~Helv.} \textbf{8} (1935--1936)
305--353.

\end{thebibliography}
\end{document}